\documentclass[11pt]{amsart}

\usepackage[margin=1.5 in]{geometry}
\usepackage{amsmath, amssymb, amsthm}
\usepackage[pdftex]{color}
\usepackage{comment}
\usepackage[colorlinks,citecolor=blue,linkcolor=red]{hyperref}
\usepackage{xcolor}
\usepackage{marginnote}
\usepackage{hyperref}
\usepackage[all,arc]{xy}

\newtheorem{thm}{Theorem}[section]

\newtheorem{question}[thm]{Question}

\newtheorem{theorem}{Theorem}[section]
\newtheorem{lemma}[theorem]{Lemma}
\newtheorem{proposition}[theorem]{Proposition}
\newtheorem{corollary}[theorem]{Corollary}

\theoremstyle{remark}

\theoremstyle{definition}

\newcommand{\ev}{\mathrm{ev}}

\newcommand{\mc}{\mathcal}

\newcommand{\nc}{\newcommand}
\nc{\dmo}{\DeclareMathOperator}

\nc{\R}{\mathbb{R}}
\nc{\Z}{\mathbb{Z}}
\nc{\N}{\mathbb{N}}
\nc{\cS}{\mathcal{S}}
\nc{\iso}{\cong}
\dmo{\Diff}{Diff}
\dmo{\Homeo}{Homeo}
\dmo{\dist}{dist}
\dmo\BDiff{BDiff}
\dmo\SO{SO}
\dmo\slide{sl}
\dmo\im{im}
\dmo\id{id}
\dmo\Fix{Fix}
\dmo\Out{Out}
\dmo{\T}{\mathcal{T}}
\dmo{\Te}{\mathcal{T}^{\epsilon}}
\dmo{\Me}{\mathcal{M}^{\epsilon}}

\begin{document}

\title{Equidistribution of hyperbolic groups in homogeneous spaces}

\author[I. Gekhtman]{Ilya Gekhtman}
\address{Department of Mathematics\\ 
Technion -- Israel Institute of Technology\\ 
Technion City \\
Haifa, Israel, 320003\\}
\email{\href{mailto:ilyagekh@gmail.com}{ilyagekh@gmail.com}}

\author[S.J. Taylor]{Samuel J. Taylor}
\address{Department of Mathematics\\ 
Temple University\\ 
1805 North Broad Street\ 
Philadelphia, PA 19122, U.S.A\\}
\email{\href{mailto:samuel.taylor@temple.edu}{samuel.taylor@temple.edu}}

\author[G. Tiozzo]{Giulio Tiozzo}
\address{Department of Mathematics\\ 
University of Toronto\\ 
40 St George St\\ 
Toronto, ON, Canada\\}
\email{\href{mailto:tiozzo@math.toronto.edu}{tiozzo@math.toronto.edu}}

\date{\today}

\begin{abstract}
 We prove that infinite orbits of Zariski dense hyperbolic groups equidistribute in homogeneous spaces, in the sense that the family of measures obtained by averaging along spheres in the Cayley graph converges to Haar measure.
\end{abstract}

\maketitle

\section{Introduction}

Generalizing the classical ergodic theorem of Birkhoff, ergodic theorems for actions of hyperbolic groups have long been a subject of interest, 
starting with Nevo--Stein \cite{NS} for free groups, and then Fujiwara--Nevo \cite{FN}, 
Bufetov \cite{Buf1, Buf2},  Bowen \cite{Bowen},  Bufetov--Series \cite{BS}, and Bowen--Nevo \cite{BN}, among others. 

Given a hyperbolic group $\Gamma$ with a measure preserving action on a probability space $(X, m)$, these authors consider Ces\`aro averages 
of the following type: let $S$ be a finite generating set for $\Gamma$, and let $S_n$ denote the sphere of radius $n$ in the Cayley graph of $(\Gamma, S)$. 
Then for any function $f \colon X \to \mathbb{R}$, any $x \in X$ and $N \geq 1$, one defines the averaging operator 
\begin{equation} \label{E:average}
c_N(f) := \frac{1}{N} \sum_{n \leq N} \frac{1}{\# S_n} \sum_{|w| = n} f(w^{-1} x).
\end{equation}
The most recent results in this vein are due to Bufetov--Khristoforov--Klimenko \cite{BKK} and Pollicott--Sharp \cite{PS}, 
who, for measure-preserving actions of hyperbolic groups, establish convergence of the Ces\`aro averages $(c_N(f))_{N \geq 0}$ for $f \in L^\infty(X, m)$ and 
for \emph{almost every} point $x \in X$. 
In these cases, the identification of the limit 
is a well-known open problem (see e.g. \cite{BKK}). In this paper, we prove the convergence of such Ces\`aro averages for \emph{every} 
starting point $x \in X$, provided that $X$ is a homogeneous space, and we identify the limiting measure.

\smallskip
To recall the setting of homogenous dynamics, consider a real Lie group $G$, a lattice $\Lambda<G$ and a subgroup $\Gamma<G$. The subgroup $\Gamma$ acts on $G/\Lambda$ by left multiplication. The distribution of orbits of $\Gamma$ in the homogeneous space $G/\Lambda$ has been the object of much research. When $G=\mathrm{SL}_{2}(\mathbb{R})$ and $\Gamma$ is the diagonal subgroup, the orbits of $\Gamma$ are precisely hyperbolic geodesics in the unit tangent bundle of $\mathbb{H}^2/\Lambda$. By ergodicity of the geodesic flow, almost every $x\in G/\Lambda$ (with respect to Riemannian
measure $m$ on $\mathbb{H}^2/\Lambda$) has dense orbit and is equidistributed with respect to Haar measure, i.e. for any continuous $f\colon G/\Lambda\to \mathbb{R}$ we have $\frac{1}{T} \int^T_{0}f(g_tx) \ dt \to \int f dm$.
Nevertheless some orbits of $\Gamma$ are closed geodesics, and others are very wild: indeed, for any $c\in [1,2]$ there is an orbit whose image in $\mathbb{H}^2/\Lambda$ has closure of Hausdorff dimension $c$. When $\Gamma<\mathrm{SL}_{2}(\mathbb{R})$ is instead the (one parameter continuous) upper triangular subgroup, the orbits of $\Gamma$ are horocycles in $\mathbb{H}^2/\Lambda$ and their orbits are much more regular: indeed, Ratner's celebrated theorem implies any $\Gamma$-orbit is either closed or dense in $G/\Lambda$ \cite{Ratner}. In the latter case, their orbits are equidistributed with respect to Haar measure. 

\smallskip
Here we are concerned with finitely generated subgroups $\Gamma<G$, in particular ones which are word hyperbolic.  We will show that under suitable assumptions, Ces\`aro averages of spheres in their Cayley graphs  become equidistributed with respect to the Haar measure on $G/\Lambda$. 
Assume that $G$ is connected semisimple and $\Gamma$ is Ad-Zariski dense, meaning that the image of $\Gamma$ by the adjoint representation 
$\mathrm{Ad} \colon G \to \mathrm{GL}(\mathfrak{g})$ is Zariski dense.
A breakthrough of Benoist and Quint \cite{BQ1, BQ2, BQ3} implies that every $\Gamma$ orbit in $G/\Lambda$ is either finite or dense.

If $x\in G/\Lambda$ is such that $\Gamma x$ is infinite, we prove that its orbit equidistributes with respect to the Haar measure $\nu$ on $X = G/\Lambda$. 
More precisely:  

\begin{theorem} \label{th:main}
Let $G$ be a connected semisimple real Lie group without compact factors, and $\Lambda < G$ an irreducible lattice.  
Let $X := G/\Lambda$ and $\nu$ be the Haar measure on $X$.
Let $\Gamma$ be a hyperbolic group, and consider a representation $\rho \colon \Gamma \to G$ with Ad-Zariski dense image, 
which defines an action of $\Gamma$ on $X$. 
Fix any finite generating set $S$ of $\Gamma$, and let $S_n$ be the sphere of radius $n$ in the Cayley graph of $(\Gamma, S)$. 
Then for any continuous $f \colon X \to \mathbb{R}$ with compact support, and any $x \in X$ we have that either the orbit $\Gamma x$ is finite, or
\begin{equation} \label{E:main-1}
\frac{1}{N} \sum_{n \leq N} \frac{1}{\# S_n} \sum_{|w| = n} f(w^{-1} x)\to \int_X f \ d \nu.
\end{equation}
\end{theorem}

We remark that since spheres in the Cayley graph are symmetric, one can instead consider $f(wx)$ in the theorem statement rather than $f(w^{-1}x)$, however this would be less natural for the proof.

\smallskip

To consider more general triples $(G,\Lambda, \Gamma)$, 
the hypothesis that $\Gamma$ is Ad-Zariski dense can be replaced with the hypothesis that $\mathrm{Ad}(\Gamma) \subset \mathrm{GL}(\mathfrak{g})$ is 
Zariski connected semisimple with no compact factor and that $\Gamma x$ is dense in $X$.
In fact, we can assume more weakly that $\overline{\Gamma x}$ is connected, in which case $\nu$ is replaced by the unique invariant (Haar) probability measure $\nu_{\overline{\Gamma x}}$ on the \emph{homogeneous} space $\overline{\Gamma x}$, which exists by Benoist--Quint \cite{BQ3}. See Section \ref{sec:passing} for details.

\begin{theorem} \label{th:main-general}
Let $G$ be a real Lie group, $\Lambda < G$ a lattice, and let $\rho \colon \Gamma \to G$ be a representation of a hyperbolic group $\Gamma$ into $G$.
Suppose that the Zariski closure of $\mathrm{Ad}( \rho(\Gamma))$ is Zariski connected, semisimple, and without compact factors. 
Fix any finite generating set $S$ of $\Gamma$, and let $x \in X$ such that the orbit closure $\overline{\Gamma x}$ is connected.
Then for any continuous $f \colon X \to \mathbb{R}$ with compact support, we have
\begin{equation} \label{E:main-gen}
\frac{1}{N} \sum_{n \leq N} \frac{1}{\# S_n} \sum_{|w| = n} f(w^{-1} x)\to \int_X f \ d \nu_{\overline{\Gamma x}}.
\end{equation}
\end{theorem}

We also show that orbits along randomly chosen geodesic rays in $\Gamma$ equidistribute in $X$ (Theorem \ref{th:random_geos}); 
see Section \ref{sec:geo_equi} for details. 

\medskip
In fact, our methods apply beyond hyperbolic groups, to groups admitting a thick geodesic combing, as defined in \cite{GTT2} (see Section \ref{S:thick}). 
Such class of groups include relatively hyperbolic groups and right-angled Artin and Coxeter groups, for certain natural generating sets. 
Representations of such groups into $\textup{SL}(n, \mathbb{R})$ are a topic of considerable recent interest, especially in the context of higher Teichm\"uller theory
\cite{GW,GGKW, Zhu, ZZ1, ZZ2}.

The most general version of the theorem we prove is the following: 

\begin{theorem} \label{th:main-3}
Let $G$ be a real Lie group, $\Lambda < G$ a lattice, and let $\Gamma$ be a finitely generated group with generating set $S$, 
such that $(\Gamma, S)$ has a thick geodesic combing. 
Let $\rho \colon \Gamma \to G$ be a representation, and suppose that the Zariski closure of $\mathrm{Ad} (\rho(\Gamma))$ is Zariski connected, semisimple, and without compact factors. Let $x \in X$ such that the orbit closure $\overline{\Gamma x}$ is connected.
Then for any continuous $f \colon X \to \mathbb{R}$ with compact support, we have
\begin{equation} \label{E:main-3}
\frac{1}{N} \sum_{n \leq N} \frac{1}{\# S_n} \sum_{|w| = n} f(w^{-1} x)\to \int_X f \ d \nu_{\overline{\Gamma x}}.
\end{equation}
\end{theorem}

For instance, using \cite[Lemma 8.1]{GTT4}, the above theorem applies to the following situations:
\begin{itemize}
\item If $\Gamma$ is relatively hyperbolic with virtually abelian 
peripheral subgroups then there exists a generating set of $\Gamma$ with thick geodesic combing (see also \cite[Sections 2.3, 9]{GTT2});
\item If $\Gamma$ is a non-abelian, irreducible, right-angled Artin or Coxeter group, and $S$ is the vertex generating set, then $(\Gamma, S)$ has a thick geodesic combing
(see also \cite[Section 10]{GTT2}).
\end{itemize}

\medskip
One particularly concrete application of Theorem \ref{th:main-general} is the following. Let $M = \mathbb{H}^3/\Lambda$ be a finite volume hyperbolic $3$--manifold. It is well known by Shah \cite{Shah} and Ratner \cite{Ratner} that every totally geodesic hyperbolic plane in $M$ is either closed or dense. 
More precisely, for any $x \in G/\Lambda$ (where $G = \mathrm{PSL}_2(\mathbb{C})$ is the group of orientation preserving isometries of $\mathbb{H}^3$) the orbit $\mathrm{PSL}_2(\mathbb{R}) x$ is either closed or dense. Theorem \ref{th:main-general} implies that if $\Gamma$ is any discrete, nonelementary subgroup of $\mathrm{PSL}_2(\mathbb{R})$ and $S$ is any generating set of $\Gamma$, then either $\Gamma x$ is finite or spheres $S_n$ in the Cayley graph for $(\Gamma,S)$ equidistribute; i.e. 
averages of the counting measures on $S_n x$ converge to the invariant (Haar) measure on $G/\Lambda$.

\medskip
Another application of our techniques is to actions on tori, where in fact we do not need to take the Ces\`aro average to guarantee convergence.

\begin{theorem} \label{th:tori}
Let $\Gamma<\mathrm{SL}(d, \mathbb{Z})$ be a Zariski dense hyperbolic group. Let $x\in \mathbb{T}^d$ be any irrational point. Then for any continuous $f : \mathbb{T}^d \to \mathbb{R}$ we have 
\begin{equation} \label{E:main-torus}
\frac{1}{\# S_n} \sum_{|w| = n} f(w^{-1} x)\to \int f dm 
\end{equation}
where the integral is taken with respect to Haar measure on $\mathbb{T}^d$.
\end{theorem}

Theorem \ref{th:tori} follows by using the recent work of He--de Saxc\'e  \cite{HdS} (extending work of \cite{BFLM}) in place of \cite{BQ3}. 
In fact, it extends to more general nilmanifolds, using \cite{HLL}. 
This leads us to ask the following question:

\begin{question}
In the context of Theorem \ref{th:main}, for which representations $\rho\colon \Gamma \to G$ can the Ces\`aro average in eq. \eqref{E:main-1} be removed?
\end{question}

This is closely related to the well-known question of Benoist--Quint \cite[Question 3]{BQ4} concerning whether the Ces\`aro average appearing in Theorem \ref{th:BQ} can be removed. Although partial progress has been made in \cite{B}, 
the hypotheses there are incompatible with the case of interest here; the measures $\mu_j$ appearing in the proof of Lemma \ref{lem:markov_prim} have the property that distinct convolution powers have disjoint support.

\subsection*{Acknowledgements} 

S. Taylor is partially supported by NSF grant DMS-2102018 and the Sloan Foundation. 
G. Tiozzo is partially supported by NSERC RGPIN-2017-06521 and an Ontario Early Researcher Award.

\section{Geodesic combings}
\label{sec:background}
 
In this section, we recall some basic properties of graph structures and geodesic combings of groups. 
For hyperbolic groups, the essential features are due to Cannon \cite{Cannon} and Calegari--Fujiwara \cite{CF}.
For the general case, we refer to \cite{GTT2}, \cite{GTT4}. 

Fix a finitely generated group $\Gamma$ and any finite subset $S \subset \Gamma$, which we usually take to be a generating set of $\Gamma$.
 A \emph{graph structure} for $(\Gamma, S)$ is a triple $(D,v_0,\ev)$, where 
$D$ is a finite directed graph, $v_0$ is a vertex of $D$ which we call its \emph{initial vertex}, and $\ev \colon E(D) \to S\subset G$ is a map that labels the edges of $D$ with elements from $S$. 
We extend the map $\ev$ by defining for each finite (always directed) path $g = g_1 \dots g_n$ the group element 
$\ev(g) = \ev(g_1) \dots \ev(g_n)$ in $G$. 
To simplify notation, we will use 
$\overline{g} = \ev(g)$ to denote the group 
element associated to the path $g$. Additionally, if there is an action $\Gamma \curvearrowright X$, we write $gx$ to mean $\overline{g}x$ for a path $g$ in $D$.

For a graph structure $D$, we define $\Omega$ to be the set of all infinite paths starting at any vertex of $\Gamma$. By $\Omega^n$ we mean the set of all paths of length $n$ and set $\Omega^* = \cup_{n\ge1} \Omega^n$. Further, if $v$ is a vertex of $D$, then $\Omega_v$ (or $\Omega_i$ if $v =v_i$) is the set of infinite paths starting at $v$, and similarly for $\Omega_v^n$.

The graph structure $D$ is \emph{geodesic}  if the map $\ev \colon \Omega_{v_0}^* \to \Gamma$ is injective and length preserving when $\Gamma$, or more precisely the subgroup generated by $S$, is given the word metric for the generating set $S$. If it is also surjective, then $D$ is said to be a \emph{geodesic combing}.
In this case, evaluation induces a bijection from $\Omega_{v_0}^n$ to $S_n$ for each $n\ge1$, where $S_n$ is the sphere of radius $n$ with respect to $S$.
In this paper, each graph structure will come from starting with a geodesic combing $D$ for $\Gamma$ and applying one or both of the following operations:
\begin{enumerate}
\item restrict the evaluation map to some subgraph $D'$; if the subgraph does not contain $v_0$ then choose an arbitrary vertex of $D'$, or
\item replace the graph $D$ with its associated $p$-step graph structure $D_p$. The vertices of $D_p$ are equal to those of $D$ and each edge of $D_p$ (and its label) naturally corresponds to a path of length $p$ in $D$.
\end{enumerate}
We observe that if $D$ is any geodesic graph structure, then so are each of $D'$ and $D_p$ as defined in items $(1)$ and $(2)$ above.

According to Cannon \cite{Cannon}, for any hyperbolic group and any finite generating set there is an associated geodesic combing. 

\subsubsection*{Structure of geodesic combings for hyperbolic groups}
We define two vertices $v_i, v_j$ to be \emph{equivalent} if there is a path from $v_i$ to $v_j$ and a path from $v_j$ to $v_i$, 
and the \emph{(recurrent) components} of $D$ as the equivalence classes for this relation. 

We denote by $A$ the transition matrix for $D$. By Perron--Frobenius, $A$ has a real eigenvalue of largest modulus, 
which we will denote by $\lambda$. 
Following Calegari--Fujiwara \cite{CF}, we say that 
the matrix $A$ is
 \emph{almost semisimple} if for any eigenvalue of maximal modulus, its geometric and algebraic multiplicity agree. For example, Calegari--Fujiwara prove that when $D$ is a geodesic combing of a hyperbolic group $A$ always satisfies this property.
 We additionally call a geodesic combing (or more generally a graph structure) \emph{semisimple} or \emph{primitive} if its transition matrix has those properties. Recall that a matrix is \emph{semisimple} if its only eigenvalue of maximal modulus is real positive and \emph{primitive} if it has a positive power. In general, primitive $\implies$ semisimple $\implies$ almost semisimple.

Let $D$ be almost semisimple, and let $\lambda$ be the leading eigenvalue of $A$. Then we say a vertex $v$ is 
of \emph{large growth} if 
$$\lim_{n \to \infty} \frac{1}{n} \log \# \{ \textup{paths of length }n\textup{ starting at }v \} = \lambda$$
and of \emph{small growth} otherwise (in which case the limit above is $< \lambda$). 
Furthermore, a component $C$ of $D$ is \emph{maximal} if 
$$\lim_{n \to \infty} \frac{1}{n} \log \# \{ \textup{paths of length }n\textup{ inside }C \} = \lambda.$$

The component-wise structure of $D$ is as follows: there is no path between maximal components and vertices of large growth are precisely the ones which have a path to a maximal component. See \cite{CF} or \cite{GTT2}.

\subsubsection*{Loop semigroups and thickness}\label{S:thick} Given a vertex $v$, we denote as $D_v$ the \emph{loop semigroup} of $v$, i.e. the set of all finite paths from $v$ to itself. 
This is a semigroup under concatenation, and all its elements lie entirely in the component of $v$. The evaluation map embeds $D_v$ into $G$ as a semigroup which we denote by $\Gamma_v$ .
Abusing terminology slightly, we also refer to $\Gamma_v$ as the loop semigroup.

Finally, we recall that any geodesic combing of a hyperbolic group has a fundamental property which we call \emph{thickness}: for any vertex $v$ in a maximal component, there is a finite set $B \subset \Gamma$ such that $\Gamma = B \cdot \Gamma_v \cdot B$. Here,  the equality is in the group $G$. See \cite[Lemma 8.1]{GTT4}
and the references therein. In general, any thick geodesic combing of a finitely generated group is automatically almost semisimple \cite[Lemma 2.3]{GTT4}.

We conclude by remarking that if $D$ is a thick geodesic combing, then the graph structures obtained by either restricting to a subgraph $D'$ of maximal growth or taking the $p$-step graph structure $D_p$ are themselves thick.
See \cite[Section 7]{GTT4} for details.

\section{Random walks on $\Gamma$ and passing to loop semigroups}
\label{sec:passing}
Let $G$ be a real Lie group and $\Lambda$ a lattice in $G$. Let $\Gamma$ be a subsemigroup of $G$ which is generated by the support of a Borel probability measure $\mu$. 

A closed subspace $Y \subset G/\Lambda$ is called \emph{homogeneous} if its stabilizer $G_Y \le G$ acts transitively on $Y$. If, in addition, $G_Y$ preserves a Borel probability measure, $Y$ is said to have \emph{finite volume}. Such a measure is unique and is denoted by $\nu_Y$. If the subsemigroup $\Gamma$ is a subgroup of $G_Y$, then $Y$ is $\Gamma$--invariant and if the action $\Gamma \curvearrowright (Y,\nu_Y)$ is ergodic, then $Y$ is called $\Gamma$--ergodic. 

\smallskip
The following theorem is due to Benoist--Quint in the case where $\mu$ is compactly supported. The generalization stated here, required for our application, is due to B\'enard--de Saxc\'e.

\begin{theorem}[Benoist--Quint \cite{BQ3}, B\'enard--de Saxc\'e \cite{BdS} Theorem C ]
\label{th:BQ}
Suppose that the Zariski closure of $\mathrm{Ad }(\Gamma) \le \mathrm{GL}(\mathfrak{g})$ is Zariski connected and semisimple with no compact factors. Further assume that the measure $\mu$ has finite first moment. Then
\begin{enumerate}
\item The orbit closure $Y = \overline{\Gamma x} \subset G/\Lambda$ is a $\Gamma$--invariant ergodic finite volume closed homogeneous subspace.
\item The sequence of measures $\left(\frac{1}{n} \sum_{k=0}^{n-1} \mu^{*k} *\delta_x \right)_{n\ge 1}$ converges to $\nu_Y$ in the weak--$*$ topology.
\item For $\nu^{\otimes \mathrm{N}^*}$--almost every sequence $(g_i)_{i\ge 1}$, the sequence of empirical measures $\left(\frac{1}{n} \sum_{k=0}^{n-1} \delta_{g_k \ldots g_1 x} \right)_{n\ge 1}$ converges to $\nu_Y$ in the weak--$*$ topology. 
\end{enumerate}
\end{theorem}

The following lemma allows us to pass conditions from $\Gamma$ to loop semigroups of the geodesic combing. This is a fundamental step toward applying Theorem \ref{th:BQ} in the proof of Theorem \ref{th:main-3}.

\begin{lemma}
\label{lem:usethis}
Suppose that $\Gamma$ is a group satisfying the hypotheses of Theorem \ref{th:main-general} (or more generally Theorem \ref{th:main-3})
and that $\Gamma_v$ is a subsemigroup of $\Gamma$ with the property that the group $\Gamma_v^\pm$ generated by $\Gamma_v$ has finite index in $\Gamma$. Then
\begin{enumerate}
\item the Zariski closures of $\mathrm{Ad}(\Gamma_v)$ and $\mathrm{Ad}(\Gamma)$ in $\mathrm{GL}(\mathfrak{g})$ are equal, and
\item if $Y = \overline{\Gamma x} \subset G/\Lambda$ is connected, then $Y$ is also the orbit closure of $\Gamma_v  x$. 
\end{enumerate}
\end{lemma}

\begin{proof}
For item $(1)$, first note that by Goldsheid--Margulis \cite[Lemma 3.3]{GM}, the Zariski closure $\mc Z(\Gamma_v)$ is a group, which implies 
$\mc Z(\Gamma_v) \supseteq \Gamma_v^\pm$.  Hence, $\mc Z(\Gamma_v) = \mc Z( \Gamma_v^\pm)$ and since $\Gamma_v^\pm$ is finite index in $\Gamma$, 
the subgroup $\mc Z( \Gamma_v)$ is finite index in $\mc Z( \Gamma)$ and thus a finite union of components. But since $\mc Z( \Gamma)$ is Zariski connected, this implies that $\mc Z(\Gamma_v) = \mc Z( \Gamma)$ as claimed.

For item $(2)$, let $Y_v, Y_v^\pm, Y$ be the orbit closures of $\Gamma_v, \Gamma_v^\pm, \Gamma$, respectively, based at $x$. By Theorem \ref{th:BQ} and the first item, each of these is a finite volume homogeneous space. 
Since $G_{Y_v}$ is a group containing $\Gamma_v$, it also contains $\Gamma_v^\pm$, hence $\Gamma_v^\pm x \subseteq G_{Y_v} x$ 
and by taking the closures $Y_v^\pm \subseteq Y_v$. Since $Y_v \subseteq Y_v^\pm$ by definition, we obtain $Y_v = Y_v^\pm$, 
hence also $G_{Y_v^\pm} = G_{Y_v}$.

Write $\Gamma = \bigcup_{b\in B} b\Gamma_v^\pm$ for a finite set $B \subseteq \Gamma$, which we can assume to contain the identity, so that $Y = \bigcup_{b \in B} b  Y_v$. Hence the smooth properly embedded submanifold $Y$ is a finite union of smooth properly embedded submanifolds diffeomorphic to $Y_v$, and so each 
$b  Y_v$ is the union of connected 
components of $Y$. 
When $Y$ is connected, we conclude that $Y = Y_v$ and $G_Y = G_{Y_v}$. 
\end{proof}

We will sometimes use the notation
$$\overline{f}(x) := \int f \ d \nu_{\overline{\Gamma x}},$$
which we observe determines a $\Gamma$--invariant function.  

\begin{lemma} \label{L:conv-RW}
Let $\mu$ be a generating measure on $\Gamma < G$, and suppose that the Zariski closure of $\mathrm{Ad}(\Gamma)$ is Zariski connected and semisimple
without compact factors. 
Let  $w_n = g_1 \dots g_n$ be the right random walk driven by $\mu$. 
Let $f \colon X \to \mathbb{R}$ be continuous, compactly supported. Then for any $g, h \in \Gamma$ and any $x \in X$, we have 
$$\frac{1}{N} \sum_{n \leq N} f(g w_n^{-1} h x) \to \overline{f}(x)$$
for almost every $(w_n)$. 
\end{lemma}

\begin{proof}
We apply Theorem \ref{th:BQ} to the measure $\check{\mu}(g) := \mu(g^{-1})$. Then a sample 
path for the left random walk driven by $\check{\mu}$ is given by 
$h_n \dots h_1 = (g_n)^{-1} \dots (g_1)^{-1} = (g_1 \dots g_n)^{-1}$
where $g_1 \dots g_n$ is a sample path for the right random walk driven by $\mu$. Hence by Theorem \ref{th:BQ}, 
for any $y \in X $ and any $\varphi$ we have
$$\frac{1}{N} \sum_{n \leq N}  \varphi(w_n^{-1} y)  \to \int \varphi \ d \nu_{\overline{\Gamma y}}.$$
Then apply the above equation with $y = h x$, $\varphi(x) = f(g x)$, using that 
the action of $g$ is measure-preserving and that $\overline{\Gamma x} = \overline{\Gamma y}$.
\end{proof}

\section{Convergence for Markov chains} \label{S:prim}
Once and for all, let us fix a countable group $\Gamma$. Throughout, we consider various thick, geodesic graph structures for $\Gamma$ whose properties are weakened over the next few sections, culminating in Section \ref{sec:general} where arbitrary thick geodesic combings are considered. 

We also fix the hypotheses of Theorem \ref{th:main-3}. That is,
\begin{itemize}
\item $G$ is a real Lie group, $\Lambda < G$ is a lattice, $X = G/\Lambda$, and $\Gamma$ is a finitely generated group with generating set $S$,
\item $\rho \colon \Gamma \to G$ is a representation, inducing an action $\Gamma \curvearrowright X$, such that the Zariski closure of $\mathrm{Ad} (\rho(\Gamma))$ is Zariski connected, semisimple, and without compact factors, 
\item $x \in X$ is a point such that the orbit closure $\overline{\Gamma x}  \subset X$ is connected.
\end{itemize}
We also set $\overline{f}(x) := \int f \ d \nu_{\overline{\Gamma x}}$, where $\nu_{\overline{\Gamma x}}$ is the Haar measure as in Theorem \ref{th:BQ}. 
If $\Gamma x$ is dense, we also write $\nu_{\overline{\Gamma x}}$ as $\nu_X$.

\smallskip

In this section, we let $D$ be a thick, geodesic graph structure for $\Gamma$ which is primitive, i.e. that its transition matrix $A$ has a positive power.
Note that we do not assume that the evaluation map is surjective.

Let $(p_i)$ be a right eigenvector for $A$, and $(q_i)$ be a left eigenvector, normalized so that $\sum_i p_i q_i = 1$.
Then we define the stationary measure as $\pi_i = p_i q_i$, and for any word $w$ of length $n$ from vertex $i$ to vertex $j$ we define
 \[
 \mu(w) = \frac{q_i p_j}{\lambda^{n}}.
 \]
Moreover, let $\mathbb{P}$ be the Markov measure on $\Omega$ whose stationary measure is $(\pi_i)$ and such that the transition probability from $v_i$ to $v_j$ is $\frac{A_{ij}p_j}{\lambda p_i}$. That is, if $W$ is the set of paths starting with a fixed prefix $w$, then $\mathbb{P}(W) = \mu(w)$.
 
Let $\Omega^n_{i,j}$ 
be the set of paths of length $n$ from $i$ to $j$. 
For any vertices $v_i, v_j$ and any $N \geq 0$, we define the modified Markov averaging operator
$$c_N^{\mu, i, j}(f) := \frac{1}{N} \sum_{n \leq N}  \sum_{w \in \Omega^n_{i,j}} \mu(w) f(w^{-1} x).$$

\begin{lemma}
\label{lem:markov_prim}
Let us consider a primitive graph structure on $\Gamma$. 
Then for any continuous, compactly supported $f$ on $X$ and for any vertices $v_i, v_j$, 
$$c_N^{\mu, i, j}(f) \to \pi_i \pi_j \overline{f}(x).$$
\end{lemma}

\begin{proof}
Let us fix a vertex $v_j$ of the graph. Then we can decompose almost every path $\gamma \in \Omega$ as 
$$\gamma = \alpha \cdot g_1 \cdot g_2 \cdot \ldots \cdot g_n \cdot \ldots$$
where $\alpha$ does not pass through $v_j$ except at its end, 
and each $g_i$ is a loop based at $v_j$. Let $\mu_j$ be the measure induced by $\mu$ on 
the loop semigroup $\Gamma_j$ associated to $v_j$.

Thus, we have, up to a set of $\mathbb{P}$-measure zero, the decomposition 
$$\Omega = \bigsqcup_{\alpha} \alpha \cdot (\Gamma_j)^\mathbb{N}.$$
For each $\alpha$, the conditional measure on the space $(\Gamma_j)^\mathbb{N}$ is 
the product measure $(\mu_j)^\mathbb{N}$, hence $w_n := g_1 g_2 \dots g_n$ 
is a random walk driven by $\mu_j$. Let $\mathbb{P}_j$ be the distribution of $(w_n)$. 

From now on, let us fix a vertex $v_j$. 
Then by Lemma \ref{lem:usethis} we have $\overline{\Gamma x} = \overline{\Gamma_j x}$ since thickness implies that the group generated by $\Gamma_j$ is finite index in $\Gamma$.
Hence by Lemma \ref{L:conv-RW} we have
$$\frac{1}{N} \sum_{n \leq N} f( w_n^{-1} \alpha^{-1} x) \to \int f \ d\nu_{\overline{\Gamma x}}$$
for $\mathbb{P}_j$-almost every $w_n$. 

Let now $R(n, j) : \Omega \to \mathbb{N}$ be the $n$th return time to $v_j$ (it depends on the infinite path, but we will omit that 
dependence in the notation). 
Since by construction $\alpha \cdot w_n = \gamma_{R(n, j)}$, we have 
\begin{equation} \label{E:conv-vertex}
\frac{1}{N} \sum_{n \leq N} f(\gamma^{-1}_{R(n, j)} x) \to \int f \ d\nu_{\overline{\Gamma x}}
\end{equation}
$\mathbb{P}$-almost surely. 

Let $T_j(N) := \max \{ k \ : \ R(k, j) \leq N \}$. 
Then, if $[\gamma_n]$ denotes the end vertex of the path $(\gamma_n)$, we obtain 
\begin{align*}
\frac{1}{N} \sum_{n \leq N} f(\gamma_n^{-1} x) \chi_{\{ [\gamma_n] = j \}} 
& =  \frac{1}{N} \sum_{k \leq T_j(N)} f(\gamma_{R(k, j)}^{-1} x) \\
&  =   \frac{T_j(N)}{N} \cdot \frac{1}{T_j(N)} \sum_{k \leq T_j(N)} f(\gamma_{R(k, j)}^{-1} x)
\end{align*}
hence by \eqref{E:conv-vertex},
\begin{equation} \label{E:almost-every}
\frac{1}{T_j(N)} \sum_{k \leq T_j(N)} f(\gamma_{R(k, j)}^{-1} x) \to \int  f \ d\nu_{\overline{\Gamma x}} \qquad \mathbb{P}\textup{--a.s.}
\end{equation}
Now, we also have $\mathbb{P}$--a.s. (e.g. \cite[Lemma 4.6]{GTT4})
$$\lim_{N}  \frac{T_j(N)}{N} = \pi_j$$
hence we have 
$$\frac{1}{N} \sum_{n \leq N} f(\gamma_n^{-1} x) \chi_{\{ [\gamma_n] = j\}}   \to  \pi_j \int  f \ d\nu_{\overline{\Gamma x}} \qquad \mathbb{P}\textup{-a.s.}$$
Then, if we integrate over all paths that start in $v_i$, 
$$\int d \mathbb{P}(\gamma) \ \frac{1}{N} \sum_{n \leq N} f(\gamma_n^{-1} x) \chi_{\{ [\gamma_n] = j \}}  \chi_{\{ [\gamma_0] = i \}} \to  \pi_j \int  f \ d\nu_{\overline{\Gamma x}} \  \mathbb{P}(\chi_{\{ [\gamma_0] = i \}}) $$
and, since $\mathbb{P}(\chi_{\{ [\gamma_0] = i \}}) = \pi_i$, 
$$\frac{1}{N} \sum_{n \leq N} \sum_{w \in \Omega^n_{i,j}} \mu(w) f(w^{-1} x) \to \pi_i \pi_j \int  f \ d\nu_{\overline{\Gamma x}}. \qedhere$$
\end{proof}

Recall that $\Omega^n$ is the set of paths of length $n$ starting at \emph{any} vertex. 
Consider the counting operator 
$$\kappa_N(f) := \frac{1}{N} \sum_{n \leq N} \frac{1}{\#\Omega^n} \sum_{|w| = n} f(w^{-1}x)$$
where the sum is over all paths of length $n$ in the graph. 
Given $i, j$, we define the modified counting operator 
$$\kappa_N^{i, j}(f) := \frac{1}{N} \sum_{n \leq N} \frac{1}{\#\Omega^n} \sum_{w \in \Omega^n_{i,j} } f(w^{-1} x)$$
Since the graph structure is primitive, we know that 
\begin{equation} \label{E:exp-gr}
c := \lim_{n} \frac{\# \Omega^n}{\lambda^n}
\end{equation}
exists, and $c > 0$. 

\begin{proposition}
\label{prop:prim_case}
Suppose that the graph structure is primitive. Then for any two vertices $v_i$ and $v_j$ and any compactly supported $f\colon X \to \mathbb{R}$, we have 
$$\lim_{N \to \infty} \kappa_N^{i, j}(f) = \frac{q_i p_j}{c} \overline{f}(x).$$
As a consequence, we also have
$$\lim_{N \to \infty} \kappa_N(f) = \overline{f}(x).$$
\end{proposition}

\begin{proof}
Let us set for any $n \geq 1$
$$a_n :=  \sum_{w \in \Omega^n_{i,j}} \mu(w) f(w^{-1} x),  \qquad b_n := \frac{\lambda^{n}}{p_i q_j \#  \Omega^n }$$
so that 
$$a_n b_n  = \frac{1}{\# \Omega^n} \sum_{w \in \Omega^n_{i,j}} f(w^{-1} x).$$
Now, we have by Lemma \ref{lem:markov_prim}
$$\frac{1}{N} \sum_{n \leq N} a_n \to \pi_i \pi_j \overline{f}(x)$$
and by \eqref{E:exp-gr}
$$b_N \to \frac{1}{c p_i q_j }.$$
Hence, as in \cite[Proposition 8]{Buf1},
$$\kappa_N^{i, j}(f) = \frac{1}{N} \sum_{n \leq N} a_n b_n \to \frac{ \pi_i \pi_j }{c p_i q_j} \overline{f}(x) = \frac{q_i p_j}{c} \overline{f}(x).$$
Then, by summing over all $i, j$, 
$$\lim_{N \to \infty} \kappa_N(f) = \alpha  \overline{f}(x)$$
with $\alpha =  \frac{\sum_i q_i \sum_j p_j}{c}$
is a constant which does not depend on $f$ or $x$. 
To see that $\alpha = 1$, we note that $\int \kappa_N(f) \ d\nu_{\overline{\Gamma x}} = \int f \ d\nu_{\overline{\Gamma x}}$ for each $N\ge 1$. 
This completes the proof. 
\end{proof}

\section{From primitive to semisimple graph structures}
In this section, we now allow the thick, geodesic graph structure $D$ for $\Gamma$ to be semisimple and generalize the results from the previous section.

Recall that $\Omega^n_0 := \Omega^n_{v_0}$ denotes the set of paths of length $n$ from the initial vertex $v_0$. 
We now consider the operator
\begin{align}\label{eq:counting}
c_N(f) := \frac{1}{N} \sum_{n \leq N} \frac{1}{\#\Omega^n_0 } \sum_{w \in \Omega^n_0 } f(w^{-1} x),
\end{align}
and recall that we have set $\overline{f}(x) := \int f \ d \nu_{\overline{\Gamma x}}$.

\begin{proposition}
Suppose that the graph structure is semisimple. Then for any continuous $f$ with compact support we have 
$$\lim_{N \to \infty} c_N(f) = \overline{f}(x).$$
\end{proposition}

\begin{proof}
Following Pollicott--Sharp \cite{PS}, for any maximal component $\mathcal{V}$, we divide the set of all paths intersecting $\mathcal{V}$ in two subsets: for some fixed $M\ge 1$, the ones spending at most time $M$ outside the maximal component, 
and the ones spending at least time $M$ outside the maximal component.

Denote as $\Omega^n_{i,j}$
the set of paths of length $n$ between $v_i$ and $v_j$ (note that, if $v_i$ and $v_j$ both lie in the maximal component, so does the path); 
$\Omega_{i,out}^n$ is the set of paths of length $n$ from vertex $v_i$ lying outside the maximal component, and 
$\Omega_{in, i}^{n}$ is the set of paths of length $n$ from the initial vertex to vertex $v_i$ lying outside the maximal component.
Let $\Omega^n_{0,\mc V+}$ be the set of paths of length $n$ from the initial vertex 
$v_0$ that intersect the component $\mathcal{V}$. 
By abuse of notation, we shall write $i \in \mathcal{V}$ to mean that the vertex $v_i$ belongs to the component $\mathcal{V}$. 

Then if we set 
$$s_{n, \mathcal{V}}(f) := \sum_{w \in \Omega^n_{0,\mc V+}} f(w^{-1} x)$$
and 
$$c_{n, \mathcal{V}}(f) := \frac{1}{N} \sum_{n \leq N} \frac{s_{n, \mathcal{V}}(f)}{\# \Omega^n_0 }$$
we have 
$$s_{n, \mathcal{V}}(f) = \sum_{i, j \in \mathcal{V}} \sum_{a + b \leq n} \sum_{g \in \Omega^a_{in,i}} \sum_{w \in \Omega^{n-a-b}_{i, j}} \sum_{h \in \Omega^b_{j,out}} f((gwh)^{-1}x).$$
Then, let us fix $M \geq 0$. Define 
$$s_{n, < M, \mathcal{V}}(f) = \sum_{i, j \in \mathcal{V}} \sum_{a + b \leq \min\{ n, M\} }  \sum_{g \in \Omega^a_{in,i}} \sum_{w \in \Omega^{n-a-b}_{i, j}} \sum_{h \in \Omega^b_{j,out}} f((gwh)^{-1}x)$$
and 
$$s_{n, > M, \mathcal{V}}(f) = \sum_{i, j \in \mathcal{V}} \sum_{M < a + b \leq n} \sum_{g \in \Omega^a_{in,i}} \sum_{w \in \Omega^{n-a-b}_{i, j}} \sum_{h \in \Omega^b_{j,out}} f((gwh)^{-1}x).$$
Let us look at the first term. Note that the restriction of the graph structure to the component $\mathcal{V}$ is primitive; 
let $(p_i), (q_i)$ be a right and left eigenvector of the transition matrix of the subgraph corresponding to $\mathcal{V}$, as in Section \ref{S:prim}. 
Hence, for each pair of vertices $v_i, v_j$ in $\mc V$ and paths $g, h$ we have by Proposition \ref{prop:prim_case} that
$$\frac{1}{N} \sum_{m \leq N} \frac{1}{\#\Omega^m_{\mc V}}  \sum_{w \in \Omega^m_{i, j}} f((gwh)^{-1}x) \to \frac{q_i p_j}{c} \overline{f}(x)$$
where
$\Omega^n_{\mc V}$ is the set of paths of length $n$ that lie entirely inside the maximal component.
Hence, by summing over all $g, h$, 
$$ \sum_{g \in \Omega^a_{in,i}} \sum_{h \in \Omega^b_{j,out}}  \frac{1}{N} \sum_{m \leq N} \frac{1}{\#\Omega^m_{\mc V}}  \sum_{w \in \Omega^m_{i, j}} f((gwh)^{-1}x) \to \#\Omega^a_{in,i} \#\Omega^b_{j,out} \frac{q_i p_j}{c} \overline{f}(x).$$
Expanding, we have
\begin{align*}
& \frac{1}{N} \sum_{n \leq N} \frac{s_{n, < M, \mathcal{V}}(f)}{\# \Omega^n_0 } \\
 &= \frac{1}{N} \sum_{n \leq N}  \frac{1}{\# \Omega^n_0 } \sum_{i, j \in \mathcal{V}} \sum_{a + b \leq \min\{ n, M\} }  \sum_{g \in \Omega^a_{in,i}} \sum_{w \in \Omega^{n-a-b}_{i, j}} \sum_{h \in \Omega^b_{j,out}} f((gwh)^{-1}x) \\
& = \frac{1}{N} \sum_{n \leq N}  \sum_{i, j \in \mathcal{V}} \sum_{a + b \leq \min\{ n, M\} } \frac{\# \Omega^{n-a-b}_{\mc V}}{\# \Omega^n_0 } \sum_{g \in \Omega^a_{in,i}}  \sum_{h \in \Omega^b_{j,out}}
\frac{1}{\#\Omega^{n-a-b}_{\mc V}}    \sum_{w \in \Omega^{n-a-b}_{i, j}} f((gwh)^{-1}x) \\
& = \sum_{g \in \Omega^a_{in,i}}  \sum_{h \in \Omega^b_{j,out}} \sum_{i, j \in \mathcal{V}} \sum_{a + b \leq  M } \frac{1}{N} \sum_{M \leq n \leq N}  \frac{\# \Omega^{n-a-b}_{\mc V}}{\# \Omega^n_0 } 
\frac{1}{\#\Omega^{n-a-b}_{\mc V}}    \sum_{w \in \Omega^{n-a-b}_{i, j}} f((gwh)^{-1}x). \\
\end{align*}
Now, note that since the graph structure is semisimple there exists a constant $A(\mathcal{V})$ such that for any $a, b \geq 0$ 
$$\lim_{n \to \infty} \frac{\# \Omega^{n-a-b}_{\mc V}}{\# \Omega^n_0 } =  A(\mathcal{V}) \lambda^{-a-b}$$
hence by taking the limit as $N \to \infty$ 
$$c_{n, < M, \mathcal{V}}(f) := \frac{1}{N} \sum_{n \leq N} \frac{s_{n, < M, \mathcal{V}}(f)}{\# \Omega^n_0 } \to  \sum_{a + b \leq  M } \sum_{i, j \in \mathcal{V}} \#\Omega^a_{in,i} \#\Omega^b_{j,out}
A(\mathcal{V}) \lambda^{-a-b} \frac{q_i p_j}{c} \overline{f}(x).$$
This implies that for fixed $M$ there is a constant $L_{M, \mc V}$ such that for any $f$ we have
$$\lim_{N \to \infty} c_{N, < M, \mathcal{V}}(f) = L_{M,\mc V} \cdot \overline{f}(x)$$
and moreover 
$$L_\mathcal{V}:= \lim_{M \to \infty} L_{M, \mathcal{V}}$$
exists. To prove the last claim, let 
$$P_{1, n} := \{ \textup{paths }g_1 \textup{ of length }n\textup{ from initial vertex to }\mathcal{V}\}$$ 
and 
$$P_{2, n} := \{ \textup{paths }g_2 \textup{ of length }n\textup{ from }\mathcal{V}\textup{ to its complement}\}.$$ 
Now, note that there exists $d > 0$ and $\mu < \lambda$ with 
$$\# P_{1, n} \leq d \mu^n \qquad \#P_{2, n} \leq d\mu^n$$
hence 
\begin{align*}
L_{M, \mathcal{V}} &= \sum_{a + b \leq  M } \sum_{i, j \in \mathcal{V}} \#\Omega^a_{in,i} \cdot \#\Omega^b_{j,out} \cdot 
A(\mathcal{V}) \lambda^{-a-b} \frac{q_i p_j}{c}\\
& \leq \sum_{i, j \in \mathcal{V}} \frac{q_i p_j}{c} A(\mathcal{V}) \sum_{a, b \geq 0}  \left(\frac{\mu}{\lambda} \right)^{a + b}
d^2 < \infty
\end{align*}
converges as $M \to \infty$.

To estimate the second term, note that 
$$c_{N, >M, \mathcal{V}}(f) := \frac{1}{N} \sum_{n \leq N} \frac{1}{\#\Omega^n_0} \sum_{M \leq a + b \leq n} \sum_{i,j} \sum_{g \in \Omega^a_{in, i}} 
\sum_{h \in \Omega^b_{j, out}}
\sum_{w \in \Omega^{n-a-b}_{i,j}}  f((gwh)^{-1}x)$$

\begin{align*}
|c_{N, >M, \mathcal{V}}(f)|  & \leq \frac{1}{N} \sum_{n \leq N} \frac{1}{\#\Omega^n_0} \sum_{M \leq a + b \leq n} \sum_{i, j} \#\Omega^a_{in, i} \cdot  \#\Omega^b_{ j, out}
\cdot \#\Omega^{n-a-b}_{i,j}  \Vert f \Vert_{\infty}  \\
& \leq \frac{1}{N} \sum_{n \leq N} \frac{1}{\#\Omega^n_0} \sum_{M \leq a + b \leq n} \sum_{i, j} d \mu^a  \cdot d \mu^b \cdot d \lambda^{n-a-b}
  \Vert f \Vert_{\infty}  \\
& \leq \frac{1}{N} \sum_{n \leq N} \frac{1}{c_1 \lambda^n} \sum_{M \leq a + b \leq n} (\#  \mathcal{V})^2 d \mu^a  \cdot d \mu^b \cdot d \lambda^{n-a-b}
  \Vert f \Vert_{\infty}  \\
 & \leq \frac{1}{N} \sum_{n \leq N}  \sum_{M \leq a + b \leq n} \frac{(\#  \mathcal{V})^2 d^3}{c_1} \left( \frac{\mu}{\lambda} \right)^{a+b} 
  \Vert f \Vert_{\infty}  \\ 
  & \leq    \frac{(\#  \mathcal{V})^2 d^3}{c_1}  \Vert f \Vert_{\infty}   \sum_{M \leq \ell } \ell \left( \frac{\mu}{\lambda} \right)^{\ell}.
\end{align*}

Hence for any $N \geq M \geq 0$ we have 
$$|c_{N, >M}(f)|  \leq r_{M, \mathcal{V}} \cdot \Vert f \Vert_\infty$$
with 
$$\lim_{M \to \infty} r_{M, \mathcal{V}} = 0.$$
As a consequence, 
$$\lim_{N \to \infty} c_{N, \mathcal{V}}(f) = L_\mathcal{V} \overline{f}(x).$$
Indeed, for any $\epsilon > 0$ there exists $M >0 $ such that $r_{M, \mathcal{V}} < \epsilon$ and $|L_{M, \mathcal{V}} - L_\mathcal{V}| < \epsilon$.
Then 
$$c_{N, \mathcal{V}}(f) = c_{N, < M, \mathcal{V}}(f) + c_{N, > M, \mathcal{V}}(f)$$
so 
\begin{align*}
\limsup_N |c_{N, \mathcal{V}}(f) - L_{\mathcal{V}} \overline{f}(x)| & \leq \limsup_N |c_{N, < M, \mathcal{V}}(f) - L_{M, \mathcal{V}} \overline{f}(x)|  + |L_{M, \mathcal{V}} \overline{f} - L_{\mc V} \overline{f}(x)| \\
&\ \quad  \quad \quad \quad \quad +  \limsup_N |c_{N, > M, \mathcal{V}}(f)| \\
& \leq |L_{M, \mathcal{V}}  - L_{\mc V}|  \Vert f \Vert_\infty + r_{M, \mathcal{V}} \Vert f \Vert_\infty \leq 2 \epsilon \Vert f \Vert_\infty
\end{align*}
which proves the claim.

Now, we have 
$$c_N(f) = \sum_{\mathcal{V}} c_{N, \mathcal{V}}(f) + c_{N, \textup{nmax}}(f)$$
where $\mathcal{V}$ runs over all maximal components, and $c_{N, \textup{nmax}}$ takes into account all paths of length at most $N$ that do not enter 
any maximal component.
Since 
$$\lim_{N \to \infty} c_{N, \textup{nmax}}(f) = 0$$
we obtain for any compactly supported $f$,
$$\lim_{N \to \infty} c_{N}(f) =  L_\infty  \overline{f}(x)$$
with $L_\infty := \sum_{\mathcal{V}} L_{\mathcal{V}}$.
By again noting that as in Proposition \ref{prop:prim_case}, $\int c_N(f) \ d\nu_{\overline{\Gamma x}} = \int f \ d\nu_{\overline{\Gamma x}}$ we see that $L_\infty = 1$.
\end{proof}

Next, for the fixed semisimple, thick graph structure $D$, we can restrict to the subgraph $D^i$ obtained by considering only paths starting at the large growth vertex $v_i$.
This gives a new semisimple, thick graph structure and we have the corresponding counting function $c_N^i$ as in eq. \eqref{eq:counting}. Applying the previous proposition then gives

\begin{corollary} \label{C:ss}
Suppose the graph structure is semisimple.
For each large growth vertex $v_i$ and any continuous functions $f$ 
with compact support, we have 
$$\lim_{N \to \infty} c^i_N(f) = \overline{f}(x),$$
where $c^i$ is restricted to the paths that start with $v_i$.
\end{corollary}

\section{From semisimple graph structures to the general case}
\label{sec:general}

We now come to the general case of our main theorem: suppose that $(D,v_0,\ev)$ is a thick geodesic combing for the pair $(\Gamma,S)$. As previously discussed, the associated transition matrix $A$ is almost semisimple. Hence, there is a $p\ge1$ so that $A^p$ is semisimple. This leads us to consider the semisimple $p$-step graph structure $D_p$, defined in Section \ref{sec:background}, whose paths of length $n$ starting at $v_0$ are in natural bijective correspondence with the paths in $\Omega^{pn}_0$.

By Corollary \ref{C:ss} (applied to the semisimple $p$-step graph structure $D_p$),
we have that for any $h\in \Omega^r_{0,i}$ with $0\le r \le p-1$:
\begin{align}\label{eq:using_semi}
\frac{1}{N} \sum_{n \le N} \frac{1}{\#\Omega^{pn}_i} \sum_{w\in \Omega^{pn}_{i}}  f((hw)^{-1}x) \to \overline{f}(x).
\end{align}
Now we recall a trick from \cite{GTT4}.

Let us fix $0 \leq r \leq p-1$. Then we can write the counting measure on $\Omega^{pn+r}_0$, starting at the initial vertex $v_0$, 
by first picking randomly a path $g_0$ of length $r$ from $v_0$ with a certain probability $\mu$, and then picking 
a random path starting at $v_i = t(g_0)$ with respect to the counting measure on the set of paths of length $n$ starting at $v_i$.

To compute $\mu$, let us consider a path $g_0$ of length $r$ starting at $v_0$ and ending at $v_i$. Then, if $v_i$ is of large growth for $D$ (and hence also large growth for $D_p$ by \cite[Lemma 7.1]{GTT4}), 
 we define 
$$\mu(g_0) :=  \frac{e_i A_{\infty} 1}{e_0 A^r A_\infty 1},$$
and otherwise $\mu(g_0) = 0$ if the end vertex of $g_0$ has small growth. Here, $A_\infty = \lim_{n\to \infty}A^{pn}/\lambda^{pn}$, which exists since $A^p$ is semisimple.

Let $\lambda'_{pn+r}$ be the measure on $\Omega^{pn+r}_0$ given by first taking randomly a path $g_0$ of length $r$ from $v_0$ 
with distribution $\mu$ and then taking uniformly a path of length $pn$ starting from $t(g_0)$. 

We previously proved (\cite[Proof of Theorem 7.3]{GTT4})
$$\Vert \lambda'_{pn+r} - \lambda_{pn+r} \Vert_{TV} \to 0$$
as $n \to \infty$, where $\lambda_n$ is the counting probability measure on $\Omega^{n}_0$ and $\Vert \cdot \Vert_{TV}$ denotes total variation. This implies that it suffices to show that for each $r$,
\[
\frac{1}{N} \sum_{n \le N} \int f(g^{-1}x) d\lambda'_{pn+r}(g) \to \overline{f}(x)
\]
as $N \to \infty$. But 
\begin{align*}
\frac{1}{N} \sum_{n \le N} \int f(g^{-1}x) d\lambda'_{pn+r}(g) 
&= \frac{1}{N} \sum_{n \le N} \frac{1}{\#\Omega^{pn}_i}  \sum_i \sum_{h\in \Omega^{r}_{0,i}} \mu(h)\sum_{w\in \Omega^{pn}_{i}}  f((hw)^{-1}x) \\
&=  \sum_i \sum_{h\in \Omega^{r}_{0,i}}  \mu(h) \frac{1}{N} \sum_{n \le N} \frac{1}{\#\Omega^{pn}_i} \sum_{w\in \Omega^{pn}_{i}}  f((hw)^{-1}x), 
\end{align*}
which as in eq. \eqref{eq:using_semi} converges to 
\[
 \sum_i \sum_{h\in \Omega^{r}_{0,i}}  \mu(h) \overline{f}(x) = \sum_{g_0\in \Omega^r_0}\mu(g_0) \overline{f}(x) = \overline{f}(x).
\]

Applying this to each $0\le r \le p-1$, we conclude that for any continuous $f$ with compact support:
\[
\lim_{N \to \infty} \frac{1}{N} \sum_{n \leq N} \frac{1}{\#\Omega^n_0 } \sum_{w \in \Omega^n_0 } f(w^{-1} x) = \overline{f}(x).
\]
Since we have assumed that $D$ is a geodesic combing for $(\Gamma,S)$, so that $\Omega^n_0$ parameterizes the sphere $S_n$ of radius $n$ in the Cayley graph of $(\Gamma,S)$, this completes the proof Theorem \ref{th:main-3}.

The proof of Theorem \ref{th:main-general} now follows using the fact, explained in Section \ref{sec:background}, that if $\Gamma$ is a hyperbolic group and $S$ is any finite generating set of $\Gamma$, there is a thick, geodesic combing of $(\Gamma,S)$.

The proof of Theorem \ref{th:main} then proceeds by replacing Theorem \ref{th:BQ} with \cite[Corollary 1.8]{BQ3}, which applies when $G$ is semisimple and $\rho(\Gamma)$ is Ad-Zariski dense.

\begin{proof}[Proof of Theorem \ref{th:tori}]
The proof follows in a very similar way as the proof of Theorem \ref{th:main-general}, by replacing the use of Theorem \ref{th:BQ}
with the main theorem of \cite{HdS}, which establishes the limit for random walks $(w_n)$, without taking Ces\`aro averages; namely, 
$$\lim_{n \to \infty} f(w_n^{-1} x) = \overline{f}(x),$$
for any irrational $x \in \mathbb{T}^n$.
The only modifications are in Lemma \ref{lem:markov_prim} and Proposition \ref{prop:prim_case}. In Lemma \ref{lem:markov_prim}, we obtain 
$$ f(\gamma^{-1}_{R(n, j)} x) \to \int f \ d\nu_{\overline{\Gamma x}}$$
almost surely, for any $v_j$; hence, since at every step the random walk lies at some vertex $v_j$, this implies 
$$ f(\gamma^{-1}_{n} x) \to \int f \ d\nu_{\overline{\Gamma x}}$$
almost surely. Integrating on both sides yields the analog of Lemma \ref{lem:markov_prim}.
In Proposition \ref{prop:prim_case}, we already know the convergence of $(a_n)$ and $(b_n)$, which immediately implies 
the convergence of $(a_n b_n)$. The rest of the proof follows verbatim, removing the average over $n \leq N$ from 
all equations.
\end{proof}

\section{Equidistribution along random geodesics}
\label{sec:geo_equi}
For $\Gamma$ hyperbolic with finite generating set $S$, we denote by $\mathrm{PS}$ any measure in the Patterson-Sullivan class on the hyperbolic boundary $\partial \Gamma$ associated to the word metric $d_S$. In this context, one can define this class as the class of any  
limit point of the sequence of spherical averages $\nu_n := \frac{1}{\#S_n} \sum_{|g| = n} \delta_g$ in the space of measures on $\Gamma \cup \partial \Gamma$. 
See \cite{Coornaert,CF,GTT1} for definitions and details.
For a geodesic ray $\gamma = [1, \eta)$ we write $\gamma(n)$ to be element of $\Gamma$ such that $d_S(1,\gamma(n)) = n$.

\begin{theorem} \label{th:random_geos}
With notation and hypotheses as in Theorem \ref{th:main-general}, if $\Gamma  x$ is infinite and has connected closure, then for $\mathrm{PS}$--almost every $\eta \in \partial \Gamma$, there is a geodesic $\gamma = [1,\eta)$ such that 
\begin{align} \label{eq:geo_equi}
\frac{1}{N}\sum_{n=1}^N \delta_{\gamma(n)^{-1}x} \longrightarrow \nu_{\overline{\Gamma x}}.
\end{align}
\end{theorem}

\begin{proof}
Let $\Omega_{0}$ be the set of infinite paths starting at $v_0$ and let $\mathbb{P}_0$ the Markov measure starting at $v_0$, which is supported on $\Omega_{0}$. There is a map $\Omega_{0} \to \partial \Gamma$ sending each infinite path to the endpoint of the associated geodesic ray in Cay$(G;S)$ and this map pushes $\mathbb{P}_0$ forward to a Patterson--Sullivan measure $\mathrm{PS}$ on $\partial \Gamma$ (\cite{CF}, \cite[Section 5]{GTT1}). 
Here we recall that any two Patterson--Sullivan measures are absolutely continuous with bounded  Radon--Nikodym derivative.

Hence, it suffices to prove equation \eqref{eq:geo_equi} for $\mathbb{P}_0$--almost every path $\gamma \in \Omega_{0}$. This was essentially accomplished in Lemma \ref{lem:markov_prim} and we now make the idea explicit.  
Let $v_i$ be a vertex in a maximal component, and let $R(n, i) : \Omega \to \mathbb{N}$ be the $n$th return time to $v_i$, 
and $T_i(N) := \max \{ k \ : \ R(k, i) \leq N \}$. 
Now by equation \eqref{E:almost-every}, noting that its proof does not use that the structure is primitive, we have 

$$\frac{1}{T_i(N)} \sum_{k \leq T_i(N)} f(\gamma_{R(k, i)}^{-1}  x) \to \int  f \ d\nu_{\overline{\Gamma x}} \qquad \mathbb{P}_0\textup{--a.s.}$$
Now since the above is true for any vertex $v_i$ in a maximal component, we have, using also $\sum_i \lim_{N \to \infty}  \frac{T_i(N)}{N} = 1$, that
\begin{align*}
\lim_{N \to \infty} \frac{1}{N} \sum_{n \leq N} f(\gamma_n^{-1}  x) &= \lim_{N \to \infty}  \sum_i \frac{T_i(N)}{N} \cdot \frac{1}{T_i(N)} \sum_{k \leq T_i(N)} f(\gamma_{R(k, i)}^{-1}  x)  \\
&= \int f  \ d\nu_{\overline{\Gamma x}}
\end{align*}
for $\mathbb{P}_0$--almost every $\gamma \in \Omega_{0}$, as desired.
\end{proof}

\end{document}